\documentclass[sii]{ipart}

\usepackage{amsthm,amsfonts,amsmath}
\usepackage{graphicx}
\usepackage{bbm}
\usepackage{subfig}
\usepackage{lipsum}
\usepackage[numbers,square]{natbib}
\RequirePackage{hyperref}

\pubyear{0000}
\volume{0}
\issue{0}
\firstpage{1}
\lastpage{10}
\arxiv{arXiv:1508.01264 }

\startlocaldefs
\DeclareMathOperator*{\argmin}{argmin}
\newcommand*{\argminl}{\argmin\limits}
\newtheorem{prop}{Proposition}
\newtheorem{corollary}{Corollary}

\endlocaldefs

\begin{document}

\begin{frontmatter}

\title{Clinical Trial Design Using A Stopped Negative Binomial Distribution}
\runtitle{Stopped Negative Binomial}



\begin{aug}
\author{\fnms{Michelle} \snm{DeVeaux} \ead[label=e1]{michelle.deveaux@regeneron.com}}
\address{Regeneron Pharmaceuticals Inc.\\ Tarrytown, NY\\ \printead{e1}}
\and
\author{\fnms{Michael J.} \snm{Kane}\thanksref{t2}\corref{} \ead[label=e2]{michael.kane@yale.edu}}
\address{Department of Biostatistics\\ School of Epidemiology and Public Health \\ Yale University, New Haven, CT \\ \printead{e2}}
\and
\author{\fnms{Daniel} \snm{Zelterman} \ead[label=e3]{daniel.zelterman@yale.edu}}
\address{Department of Biostatistics\\ School of Epidemiology and Public Health \\ Yale University, New Haven, CT \\ \printead{e3}}

\thankstext{t2}{Corresponding author.}

\runauthor{DeVeaux, Kane, and Zelterman}

\affiliation{Yale University}
\end{aug}

\begin{abstract}
We introduce a discrete distribution suggested by curtailed
sampling rules common in early-stage clinical trials. We derive the
distribution of the smallest number of independent and identically
distributed Bernoulli trials needed to observe either $s$ successes 
or $t$ failures. This report provides a closed-form expression for the 
mass function, moment generating function, and provides connections to other,
standard distributions.
\end{abstract}



 
\end{frontmatter}

\section{Introduction and Motivation}

Consider a prototypical early phase, single-arm clinical trial in which 17 patients
are enrolled and treated. The trial is modeled as a sequence of independent Bernoulli($p$) samples. Suppose the Bernoulli probability of a patient 
responding to treatment is $p=0.2$ under the null hypothesis that the treatment is not any more effective than the current standard of care.
If seven or more patients out of these 17 respond to the treatment then we 
reject this hypothesis and claim the treatment has successfully showed superiority at 
a significance level of $0.1$.  If fewer than seven respond then the null 
hypothesis is not rejected and the treatment is said to have failed to show superiority. The trial ends when 
either seven responders or 11 non-responders are observed.

If all 17 patients are enrolled at once, as in the classic
design, then the sample size is 17. However, in most clinical trials the
patients are enrolled sequentially over time.
In the present example, observing seven
successful patients ends the trial and so the number of enrollees required
could be as small as seven. Similarly 11
observed treatment failures would also end the trial. This sampling mechanism, in
which the experiment ends as soon as either predefined endpoint is reached, is
called {\em curtailed sampling}. Under curtailed sampling, the range of the
sample size for this trial is seven through 17.

The distribution of the number of trial enrollments is shown in 
Figure~\ref{fig:kane_viz_1}. There is relatively little probability mass
for values of sample sizes from seven through 10 since $p$ is small and it 
is unlikely the treatment will succeed quickly.
Figure~\ref{fig:kane_viz_2} shows the expected value and variance for the
number of trial enrollments varying the success rate $p$ between zero and one. When $p$ is small then the treatment
is likely to fail shortly after the $11^{\text{th}}$ enrollment.
When $p$ is large then the treatment is more likely to succeed and the 
number of enrollees approaches seven from above. 

When $p=0$ or $1$ then the processes are deterministic and variance is zero.
Values of $p$ between zero and one change the mixing proportions of 
the two endpoints. When $p$ is close to zero, most of the variance
is contributed from the failure endpoint. The saddle around $p=0.25$ results 
from a trade-off between the success and failure endpoints.

In the rest of this work, we derive the distribution of the number of 
enrollees needed
to observe either $s$ successes or $t$ failures. We refer to this distribution
as the Stopped Negative Binomial (SNB). 
In Section 2 we derive the distribution function and explores its properties.
Section 3 derives the moment generating function.
Section 4 describes the likelihood function.
Section 5 describes the posterior and predictive probability distributions. Section 6 relates the distribution to other standard
distributions and connects the SNB tail probability to the binomial tail 
probability. Section 7 shows how to design a trial using the SNB in a prototypical setting.

\section{Probability Mass Function}
\label{notation.section}

Let $\,b_1, b_2, \ldots \,$ denote a sequence of independent, identically
distributed, Bernoulli random variables with $\mathbb{P}[b_i=1]=p$ and
$\mathbb{P}[b_i = 0] = 1-p$, for
probability parameter $0\leq p \leq 1$. In the clinical trial setting
$\,b_i = 1$ corresponds to the $i$th patient responding to treatment.  

Let $s$ and $t$ be positive integers.  Define the SNB random
variable $Y$ as the smallest
integer value such that $\,\{b_1, \ldots , b_Y\}\,$ contains {\em either}
$\,s\,$ responders {\em or} $\,t\,$ non-responders. That is, the SNB 
distribution of $Y$ is the smallest integer such that either
$\sum_i^Y b_i = s$ or $\sum_i^Y (1-b_i) = t$.

The sequence of Bernoulli random variables can be modeled as the process 
$\mathbf{X} = \left\{X_k : k = 0,1,... \right\}$
with $X_0=0$ and
\begin{equation*} \label{eqn:proc}
X_{k+1} = X_k + b_{k+1} \ \mathbbm{1}_{\{ k-t < X_k < s\}}.
\end{equation*}
where $\mathbbm{1}_{\{f\}}$ is the {\em indicator function}, taking the value 
of one if $f$ is true and zero otherwise, and
At each step a patient's outcome is measured. If it is a response, the response count increases, otherwise it stays the same. The process continues until either 
$X_k = s$ or $X_k = k-t$ corresponding to the success and failure boundaries

In Figure~\ref{fig:kane_viz} provides a graphical illustration of $X_k$ against
$k$ for one possible realization where $s=7$ and $t=11$. The vertical axis denotes the number of
successful outcomes. The vertical axis counts the number of responders observed.
The horizontal axis counts the number of patients 
enrolled. It is strictly increasing in sequence size and the failure boundary is tilted so that no more than $t$ (in this case 11) failure can occur. The horizontal and tilted boundaries represent the two
endpoints for the trial. In this case, the seventh response was reached on
the $15^{\text{th}}$ enrollment.
Since the success boundary is reached, we would reject the null hypothesis in this example.

We next derive the probability mass function of $Y$. The distribution of $\,Y\,$ has support on integer values in the range
\begin{equation*}               
     \min(s,t) \leq \; Y \;\leq s+t-1  \label{range.y.eq}.
\end{equation*}
In Proposition \ref{stopping_time} below, we show the probability mass function of $Y$ is
\begin{align} \label{eqn:pmf}
\mathbb{P} [Y=k] =& \ S(k, p, s) \ \mathbbm{1}_{\{s \leq k \leq s+t-1\}} \nonumber \\
  & + S(k, 1-p, t) \ \mathbbm{1}_{\{ t \leq k \leq s+t-1 \}}
\end{align}
where
\begin{equation} \label{eqn:N}
S(k, p, s) = {k-1 \choose s-1} p^s (1-p)^{k-s} 
\end{equation}
is a translated version of the negative binomial probability mass.

The negative binomial cumulative distribution function can be expressed in terms of the
{\em regularized incomplete beta function} \citep{Olver2010} 
\begin{equation*}
\mathcal{I}_{1-p}(l - s + 1, s) = \sum_{k=s}^l S(k, p, s) 
\end{equation*}
for all integer values $l\,$ satisfying $s \leq l \leq s+t-1$. It has property \citep{Uppuluri1967}
\begin{equation} \label{ribf}
\mathcal{I}_{1-p}(l-s+1, s) = 1 - \mathcal{I}_p(s, l-s+1).
\end{equation}

\begin{prop} \label{stopping_time}
The distribution of the stopping time
\begin{equation*}
Y = \argminl_k \left[X_k \geq s \cup X_k \leq k-t \right]
\end{equation*}
is given at (\ref{eqn:pmf}).
\end{prop}
\begin{proof}

The endpoint $X_k = s$ can only occur if $X_{k-1} = s-1$ followed by a success. That is, 
\begin{align} \label{prop11}
\mathbb{P} [X_k =s ] \;&=\; p \, \mathbb{P} [X_{k-1} = s-1 ] \nonumber \\
 &=\; p\, {k-1 \choose s-1} p^{s-1} (1-p)^{k-s} \nonumber \\
 &=\; {k-1 \choose s-1} p^{s} (1-p)^{k-s}.
\end{align}
This expression is given
in (\ref{eqn:N}). 
Similarly, the probability a given realization reaches the endpoint $X_k = k-t$
satisfies 
\begin{align} \label{prop12}
\mathbb{P} [X_k = k-t ] \;&=\; (1-p) \, \mathbb{P} [X_{k-1} = k-t ] \nonumber \\
 &=\; (1-p)\, {k-1 \choose t-1} (1-p)^{t-1} p^{k-t} \nonumber \\
&=\; {k-1 \choose t-1} (1-p)^{t} p^{k-t}.
\end{align}
The result (\ref{eqn:pmf}) follows by summing (\ref{prop11}) and (\ref{prop12}).

To show (\ref{eqn:pmf}) sums to one, define
\begin{equation*} 
R = \sum_{k=s}^{s+t-1} S(k, p, s) + \sum_{k=t}^{s+t-1} S(k, 1-p, t).
\end{equation*}
Substitute $i=k-s$ in the first summation and $j=k-t$ in the second. Then
$R$ can be written as the cumulative distribution function (CDF) of two
negative binomial distributions:
\begin{align*}
R = \sum_{i=0}^{t-1} &{i+s-1 \choose i} p^s (1-p)^i +
\sum_{j=0}^{s-1}  {j+t-1 \choose j} p^j  (1-p)^t \\
   &= 1-\mathcal{I}_p(s, t) + 1 - \mathcal{I}_{1-p}(t, s) \\
   &= 1
\end{align*}
using (\ref{ribf}). This completes the proof that (\ref{eqn:pmf}) is the distribution 
of the stopping time and it is a valid probability mass function.
\end{proof}

We next consider an interim analysis of a clinical trial after $s'$ 
patients have responded to treatment 
and $t'$ failed to respond for $s' < s$ and $t' < t$.
\begin{corollary} \label{conditional_distribution}
The number of subsequent enrollments needed 
to reach either $s$ or $t$ endpoints behaves as SNB($p$, $s-s'$, $t-t'$).
\end{corollary}
Having observed $s'$ responders and $t'$ non-responders, there are $s-s'$ 
additional responders needed to reach the success endpoint and $t-t'$ additional 
non-responders needed to reach the failure endpoint.

\section{The Moment Generating Function} \label{sect:mgf}

\begin{prop} Let $Y$ be distributed SNB with parameters $p$, $s$, and $t$.
Then the moment generating function (MGF) of $Y$ is
\begin{equation} \label{eqn:mgf}
\mathbb{E}~e^{xY} = \left(\frac{p e^x}{1 - qe^x}\right)^s 
  \mathcal{I}_{1-qe^x} (s, t) + \left(\frac{qe^x}{1-pe^x}\right)^t 
  \mathcal{I}_{1-pe^x}(t, s)
\end{equation}
for $q = 1-p$ and is defined for 
$x < \min \left\{\log(1/p), \log(1/q) \right\}$.
\end{prop}
The moment generating function for the SNB is calculated in a manner similar to 
that of two negative binomial distributions. Appendix 1 provides a proof for the derivation.

\begin{prop}
The MGF of the SNB converges to that of the negative binomial when either
$s$ or $t$ gets large. That is
\begin{equation*}
lim_{t\to\infty} \ \mathbb{E} e^{xY} = \left( \frac{pe^x}{1-qe^x} \right)^s
\end{equation*}
as. The analogous result holds when $s \rightarrow \infty$.
\end{prop}
\begin{proof}
The second incomplete beta function in (\ref{eqn:mgf}) can be written
in terms of a cumulative binomial distribution
\begin{equation*}
\mathcal{I}_{1-pe^x}(t, s) = \mathbb{P}\left[ B \leq s-1 \right]
\end{equation*}
where $B$ is distributed as
Binomial($t-k$, $pe^x$). From Chebychev's inequality 
it follows that
\begin{equation} \label{eqn:hoeffding}
\mathbb{P}\left[ B \leq s-1 \right] \leq 
  \frac{ (t-k) pe^x (1-pe^x) }{ \left(s - (t-k)pe^x\right)^2 }
\end{equation}
As $t$ gets large $\mathcal{I}_{1-pe^x}(t, s)$ tends to zero
and $\mathcal{I}_{1-qe^x}(s, t)$ approaches 
one. The proof follows by realizing 
\begin{equation*}
0 < \frac{qe^x}{1-pe^x} < 1
\end{equation*}
over the support of $x$.
\end{proof}

\section{The Likelihood Function} \label{sect:likelihood}

In this section we derive the likelihood function for the SNB for a single trial. In early-stage clinical trial only a single trial is performed, usually because of resource constraints, and the object of interest is $p$ parameter, which determines if a therapy will be delivered to the market. A multi-sample version is less relevant for this use case but is represented as a product of mixtures of Beta distributions. Deriving it's theoretical characteristics not straight-forward. 

Let $Y$ be distributed SNB($p, s, t$). Then, the likelihood of $Y$ is proportional to a mixture of Beta distributions.
\begin{equation*} \label{eqn:likelihood}
\mathcal{L} (p | s, t, Y) = B_1 \, \mathbbm{1}_{\{s \leq Y_1\}} + B_2 \, \mathbbm{1}_{\{t \leq Y_1 \}}
\end{equation*}
where $B_1 =$ Beta$\left(s, Y - s\right)$ and $B_2 =$ Beta$\left(Y - t, t\right)$ and
\begin{equation*}
\text{Beta}\left( \alpha, \beta \right) = \frac{\Gamma(\alpha+\beta)}{\Gamma(\alpha)\, \Gamma(\beta)} 
p^{\alpha-1} (1-p)^{\beta-1}.
\end{equation*}

\begin{prop} \label{mode_order}
The mode of $B_1$ occurs at a value of $p$ greater than that of $B_2$ in the likelihood function. 
\end{prop}
\begin{proof}
The mode of the Beta($\alpha,\, \beta$) distribution is $(\alpha-1) / (\alpha+\beta-2)$. Plugging in the shape parameters of $B_1$ and $B_2$ into the expression of the mode, the proposition is equivalent to showing
\begin{equation*}
\frac{s-1}{Y - 2} > \frac{Y - t- 1}{Y -2 }
\end{equation*}
which is true when $s > Y - t$. The maximum value of the right hand of the inequality occurs when $Y_1 = s+t-1$ and the inequality is equivalent to $s > s-1$, which is true.
\end{proof}
\begin{prop}
The difference between the modes of $B_1$ and $B_2$ is bounded by
\begin{equation*}
\textrm{MODE}(B_1) - \textrm{MODE}(B_2) 
\; \geq \; \frac{1}{s+t-3}.
\end{equation*}
\begin{proof}
Proposition \ref{mode_order} shows that the mode of $B_1$ is greater than that of $B_2$. The difference
between the two can be expressed as
\begin{equation*}
\frac{s-1}{Y-2} - \frac{Y-t-1}{Y-2} = \frac{s+t-Y}{Y-2}.
\end{equation*}
This function is strictly increasing in $Y$ over its support and obtains its minimum at $s+t-1$. The result follows.
\end{proof}
\end{prop}

As an example, the likelihood function for $Y = 7, 11, 13$ and 17 is shown in
Figure~\ref{likelihood.fig}. When $Y = 7$ the trial must have ended in success and the likelihood function concentrates near 1. The success and
failure endpoints can be reached for any value of $Y \geq 11$. When 
$Y = 11$ we see a bimodal likelihood function which one mode, at $p=0.6$
provided by the success endpoint and the other, at $p=0$, from the
failure endpoint where no responses are observed. Similarly, when $Y_1=13$
we see contributions from both the success and failure endpoints but
the two modes are converging. At $Y=17$ the endpoints contributed likelihoods with similar modes and the result is unimodal.

After an endpoint has been reached, the resulting conditional likelihood is either $B_1$ or $B_2$, depending if the trial was a success or failure. However, when the endpoint is not known, such as the planning phase of a trial, unintuitive situations may arise. Since the likelihood is bimodal, there are even settings where we may reject the null despite a poorer alternative. Suppose in the hypothetical trial $p_0=0.25$ and $p_1=0.5$, the outcome is unknown, and the trial completes after 11 enrollees, as shown in Figure~\ref{likelihood.fig} labeled $Y=11$. Since the null is at a ``trough'' in the likelihood we fail to reject even though it is closer to the most likely value of $p=0$. 

\section{The Posterior and Predictive Probability Distribution} \label{bayesian_extension}

Consider the Bayesian setting where $Y$ is an SNB($P$, $s$, $t$) distribution and the rate parameter, $P$ 
is distributed as Beta($\alpha$, $\beta$). The posterior distribution $P|Y$ is proportional to the 
likelihood, given by the function
\begin{align} \label{eqn:ptl}
f_{P|Y}(p, k, s, t, \alpha, \beta) \; \propto \ &
  \ \frac{ {k-1 \choose s-1} }{B(\alpha, \beta)} p^{\alpha +s -1} 
    (1-p)^{k+\beta-s-1} \\
  & + \frac{ {k-1 \choose t-1} }{B(\alpha, \beta)} p^{k+\alpha -t -1} 
    (1-p)^{\beta+t-1} \nonumber
\end{align}
where $0 \leq p \leq 1$ and $\min(s, t) \leq k \leq s+k-1$.

The predictive distribution of the SNB can be found as by integrating
$p$ over the interval zero to one and applying the definition of the 
beta function.
\begin{align} \label{eqn:predictive}
f_{Y}(&k, s, t, \alpha, \beta) = 
  \int_0^1 f_P(p | \alpha, \beta )  f_{Y|P}(p, k, s, t) dp \\ 
  = & {k-1 \choose s-1} \frac{B\left(\alpha+s, k-s+\beta \right)}{B(\alpha, \beta)} 
     \nonumber \\
  & + {k-1 \choose t-1} 
    \frac{B\left(\alpha + k - t, t+\beta\right)}{B(\alpha, \beta)} \nonumber
\end{align}
If both $\alpha$ and $\beta$ are non-negative integers then the predictive
distribution is a mixture of negative-hypergeometric distributions. 
\begin{align*} \label{eqn:hypergeo}
f_{Y}(k, s, t, \alpha, \beta) = & 
  \frac{ {k - 1 \choose s - 1 } 
         {\alpha + \beta \choose \alpha} }{
         {\alpha + \beta + k - 1 \choose \alpha + s}}
  \frac{\alpha}{\alpha + \beta}
  \frac{ \beta }{ k-s+\beta } \\ & +
  \frac{ {k - 1 \choose t - 1} 
         {\alpha + \beta \choose \beta} 
         }{
         {\alpha + \beta + k -1 \choose t + \beta}
       }
  \frac{ \beta }{ \alpha + \beta}
  \frac{ \alpha }{ k-t + \alpha}
\end{align*}
The ratio of combinations in the first term can be interpreted as the
probability of $s-1$ responders from $k-1$ patients in $\alpha + s$ draws
from a population size of $\alpha + \beta + k - 1$. This value is multiplied
by $\alpha / (\alpha + \beta)$, the expected response rate of the prior. 
The final term in the product weights the prior based on the number
of non-responders ($k-s$). Terms in the second summand are interpreted similarly
for non-responders.

The ratio of (\ref{eqn:ptl}) divided by (\ref{eqn:predictive}) gives the 
posterior distribution of $P$. It is a mixture of beta distributions. The
mixing parameter depend on the endpoints ($s$ and $t$), the number of enrollees needed to reach an endpoint ($k$), and the prior parameters ($\alpha$ and
$\beta$).

\section{Connections and Approximations to Other Distributions}

Examples of different 
shapes of the SNB are shown in in Figure~\ref{shapes.fig} varying parameters $p$, 
$s$, and $t$. The SNB distribution is a generalization of the negative 
binomial. As a result, the SNB can approximate other distributions in the same
manner as the negative binomial. When $t$ is large then $Y-s$ has a 
negative binomial distribution with
\begin{equation*}                                    
\mathbb{P}[Y=s+j ]        \label{nb1.eq}          
  = {{s+j-1}\choose{s-1}} p^s (1-p)^j
\end{equation*}
for $\,j=0, 1,\ldots\,$. A similar statement can be made when $s$ is large
and $t$ is moderate. As a result, with proper parameter choice, the SNB
can mimic other probability distributions in a manner similar to 
those described in \cite{Best1974} and \cite{Peizer1968}. As a generalization of the negative binomial distribution the 
SNB inherits the ability to approximate other distributions. For example,
when $s=1$ and $t \rightarrow \infty$, the SNB($p$, $s$, $t$) converges
to an negative binomial distribution with index parameter $s$ and rate parameter $p$. When $s=1$, this is the geometric distribution. The connection between the negative binomial and the gamma distribution
are well-studied in the literature (see \cite{Best1974,Ord1968,Guenther1972} for examples) as well the connection to the Poisson \cite{Anscombe1950}.

The SNB generalizes both the minimum (riff-shuffle) and maximum negative
binomial distributions up to a translation of the support.
For the special case of $\,s=t,$ the SNB distribution is the
maximum negative binomial \cite{Johnson2005,Zelterman2005,Zhang2000} - 
the smallest number of outcomes necessary to 
observe at least $s$ responders {\em and} $s$ non-responders. This is equivalent to a 
translated version of the riff-shuffle or minimum negative binomial 
distribution~\citep{Johnson2005,Uppuluri1967}.

There is also an equivalence between the probability of reaching an endpoint 
in the SNB model and the tail probability of the binomial distribution.
Specifically, the probability the number of responders is at least 
$s$ in the binomial model is the same as the probability the treatment succeeds 
(reaches $s$) in the SNB model.
\begin{prop} \label{binomial_tail}
Let $Y$ be distributed as SNB($p$, $s$, $t$) and let $X_Y$ correspond
to the number of responders at the end of the trial. Let 
$B$ be distributed binomial with index parameter $n=s+t-1$ and response 
probability $p$. Then
\begin{equation}
\mathbb{P}[B \geq s] = \mathbb{P} [X_Y = s].
\end{equation}
\end{prop}
\begin{proof}
The binomial tail probability is
\begin{equation*}
\mathbb{P}[B \geq s] = 1 - \mathcal{I}_{1-p}(s, t)
\end{equation*}
The corresponding success probability is
\begin{equation} \label{eqn:success}
\mathbb{P} [X_Y = s] 
  = \sum_{k=s}^{s+t-1} {k-1 \choose s-1} p^s (1-p)^{k-s}.
\end{equation}
Let $i=k-s$. Since
\begin{equation*}
{i+s-1 \choose s-1} = {i+s-1 \choose i},
\end{equation*}
the summation in (\ref{eqn:success}) can be rewritten as
\begin{align*}
\mathbb{P} [X_Y = s] &= \sum_{i=0}^{t-1} {i+s-1 \choose i} p^s (1-p)^i\\
  &= 1 - \mathcal{I}_{1-p}(t, s)
\end{align*}
completing the proof.
\end{proof}

To illustrate this result, let us return to our initial example
where $s=7$, $t=11$, and $p=0.2$.  The probability masses in
Figure~\ref{fig:snb_bin_compare} represented in 
black are equal in panels (a) and (b) as are the masses in grey.
The probability $s$
responders are reached in the SNB process is the same as the binomial 
probability of at least seven responders. Likewise, the probability $t$ 
non-responders reached in the SNB process is the same as the binomial
probability of zero through six responders.

\section{Trial Design with the SNB}

Consider the problem of designing a trial using curtailed sampling and the SNB. Assume that the maximum number of patients $n$ is given along with the null and alternative hypotheses. In this case, the parameter $s$ tells how many responses are needed to reach the success endpoint, can vary between 1 and $n-1$. For fixed $s$ and $n$  the value of $t$ is determined by the relation $t=n-s+1$. The ROC curves for all trial designs in the prototype, where $n=17$, $p=0.2$ under the null, and $p=0.4$ under the alternative, are shown in Figure~\ref{fig:all_sample_sizes_power_vs_signif}. In designing these trials, small significance and large power values are attained by either increasing the value of $n$ or increasing the difference between the alternative and null response rates. Figure~\ref{fig:all_sample_sizes} shows the expected number of enrollees for each of these trials as a function of the number of responses required to reach the success endpoint.

The expected number of enrollees for each trial is shown in Figure~\ref{fig:all_sample_sizes}. The curve reaches its maximum of 15 patients when $s=5$. The prototype design ($p=0.2$, $s=7$, and $t=11$) has an expected sample size of 14 enrollees.

\section*{Acknowledgements}
This research was supported by grants R01CA131301, 
R01CA157749, R01CA148996, R01CA168733, and PC50CA196530 awarded by the 
National Cancer Institute along with support from the Yale Comprehensive Cancer 
Center and the Yale Center for Outcomes Research. We would also like 
to thank Rick Landin at LaJolla Pharmaceutical for his suggestions.

\section*{APPENDIX: PROOF OF PROPOSITION 2}
The MGF of the SNB is, by definition:
\begin{align*}
\mathbb{E}~e^{xY} = & \sum_{k=s}^{s+t-1} {k-1 \choose s-1} p^s q^{k-s} e^{kx} \\
  & + \sum_{k=t}^{s+t-1} {k-1 \choose t-1} p^{k-t} q^t e^{kx}
\end{align*}
and can be rewritten as:
\begin{align} \label{eqn:first_sum}
\mathbb{E}~e^{xY} = & \sum_{k=s}^{s+t-1}{k-1 \choose s-1} (pe^x)^{s} (qe^x)^{k-s} \nonumber \\
  & + \sum_{k=t}^{s+t-1}{k-1 \choose t-1} (qe^x)^t (pe^x)^{k-t}.
\end{align}
The first summation in (\ref{eqn:first_sum}) satisfies
\begin{align*}
\sum_{k=s}^{s+t-1} &{k-1 \choose s-1}  (pe)^{sx} (qe^x)^{k-s}  \\
  & = \left(\frac{pe^x}{1 - qe^x}\right)^s \ \ \sum_{k=s}^{s+t-1} {k-1 \choose s-1} 
    (qe^x)^{k-s} (1-qe^x)^s \\
  & = \left(\frac{pe^x}{1 - qe^x}\right)^s \mathcal{I}_{1-qe^x}(s, t).
\end{align*}
Since the $p$ parameter in $\mathcal{I}_p$ has domain zero 
to one, we have $0 \leq pe^x < 1$. This implies $x < -\log(p)$.
A similar expression can be derived for the second summation in 
(\ref{eqn:first_sum}) and results in
the constraint $x < -\log(1-p)$.

\bibliographystyle{acm}
\bibliography{mybibfile}

\pagebreak

\onecolumn 

\section*{Figures}

\begin{figure}[htbp!]
\centering
\includegraphics[width=\textwidth]{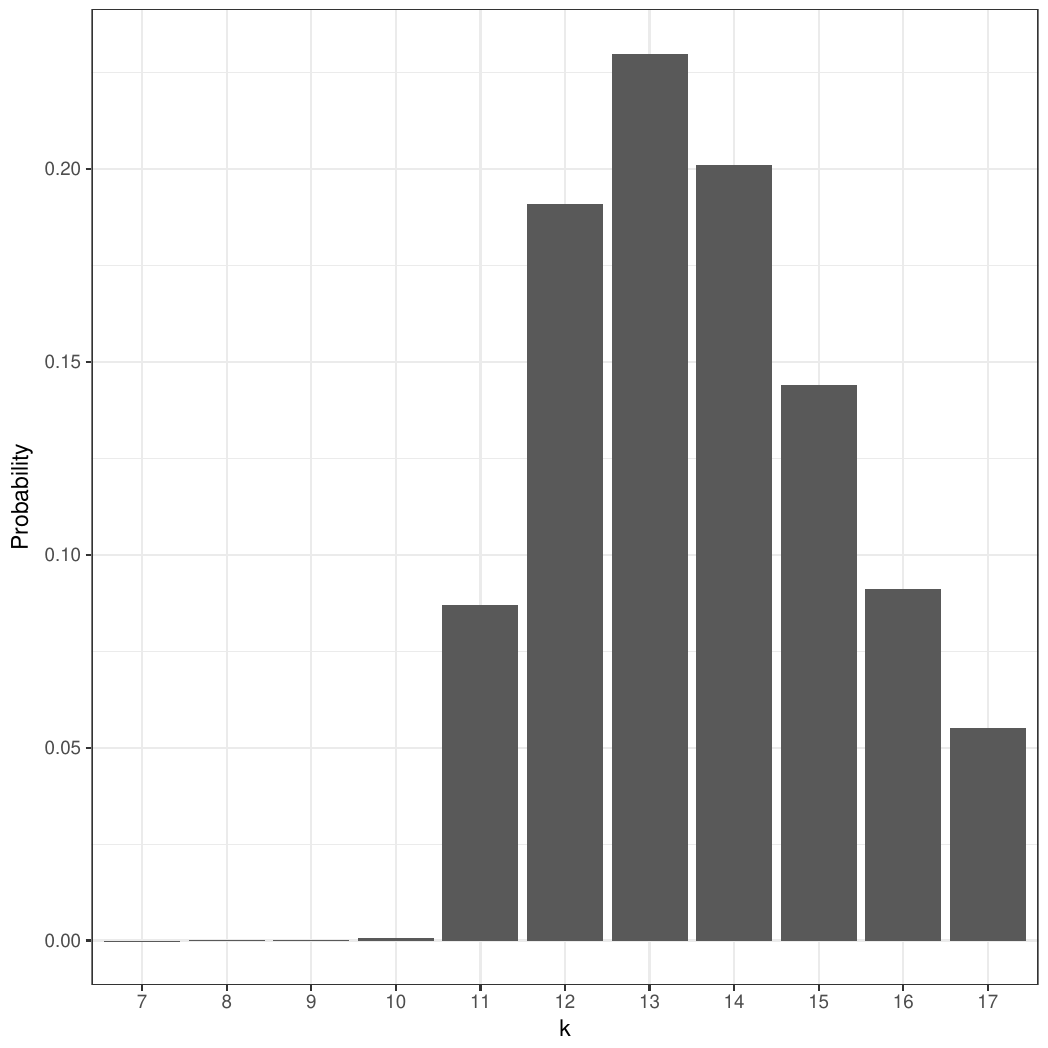}
\caption{
The distribution of the sample size in a trial that stops after either seven patients
respond to treatment or 11 do not when $p=0.2$.
}
\label{fig:kane_viz_1}
\end{figure}

\pagebreak

\begin{figure}[htbp!]
\centering
\includegraphics[width=\textwidth]{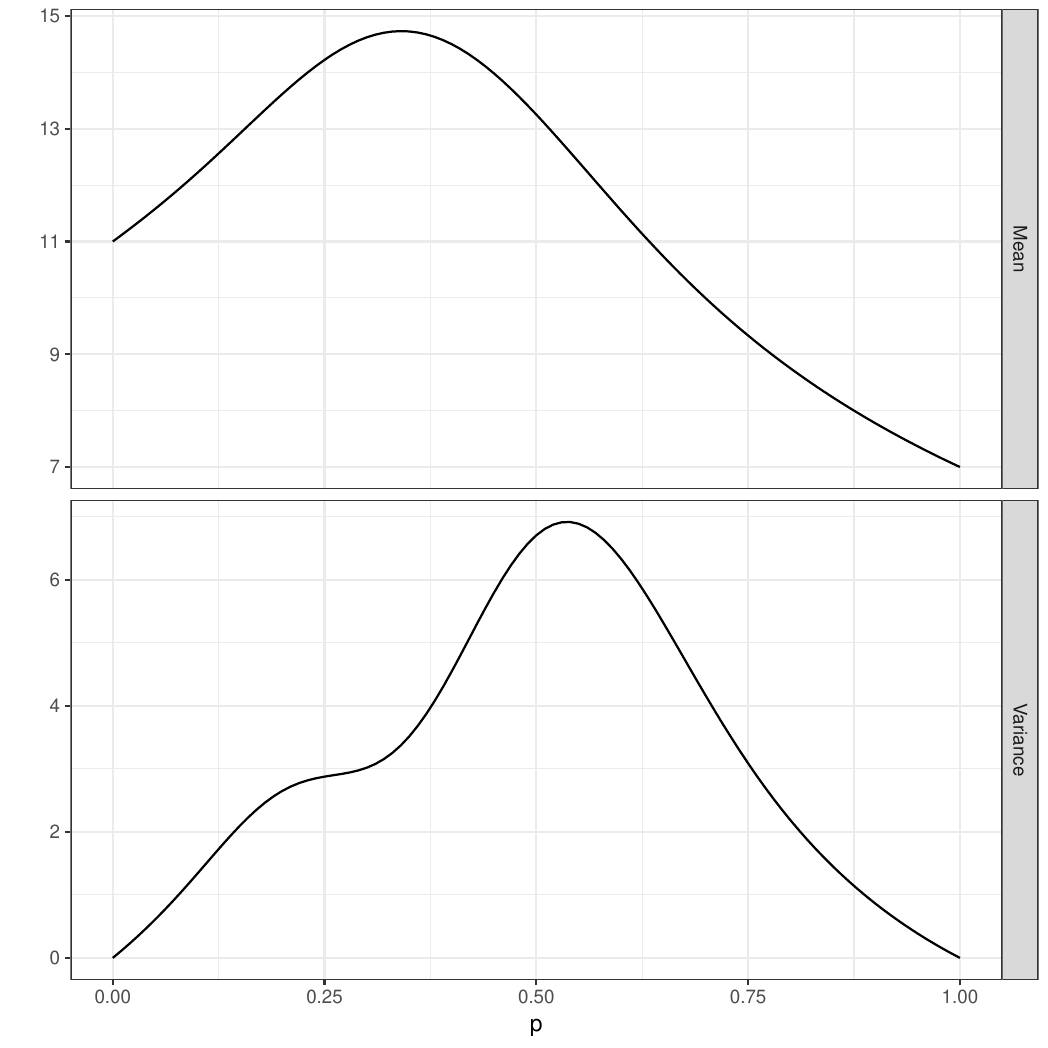}
\caption{
The mean and variance of the sample size for the probability of success $p$ 
when the trial stops after $s=7$ patients respond or $t=11$ fail to respond.
}
\label{fig:kane_viz_2}
\end{figure}

\pagebreak

\begin{figure}[htbp!]
\centering
\includegraphics[width=\textwidth]{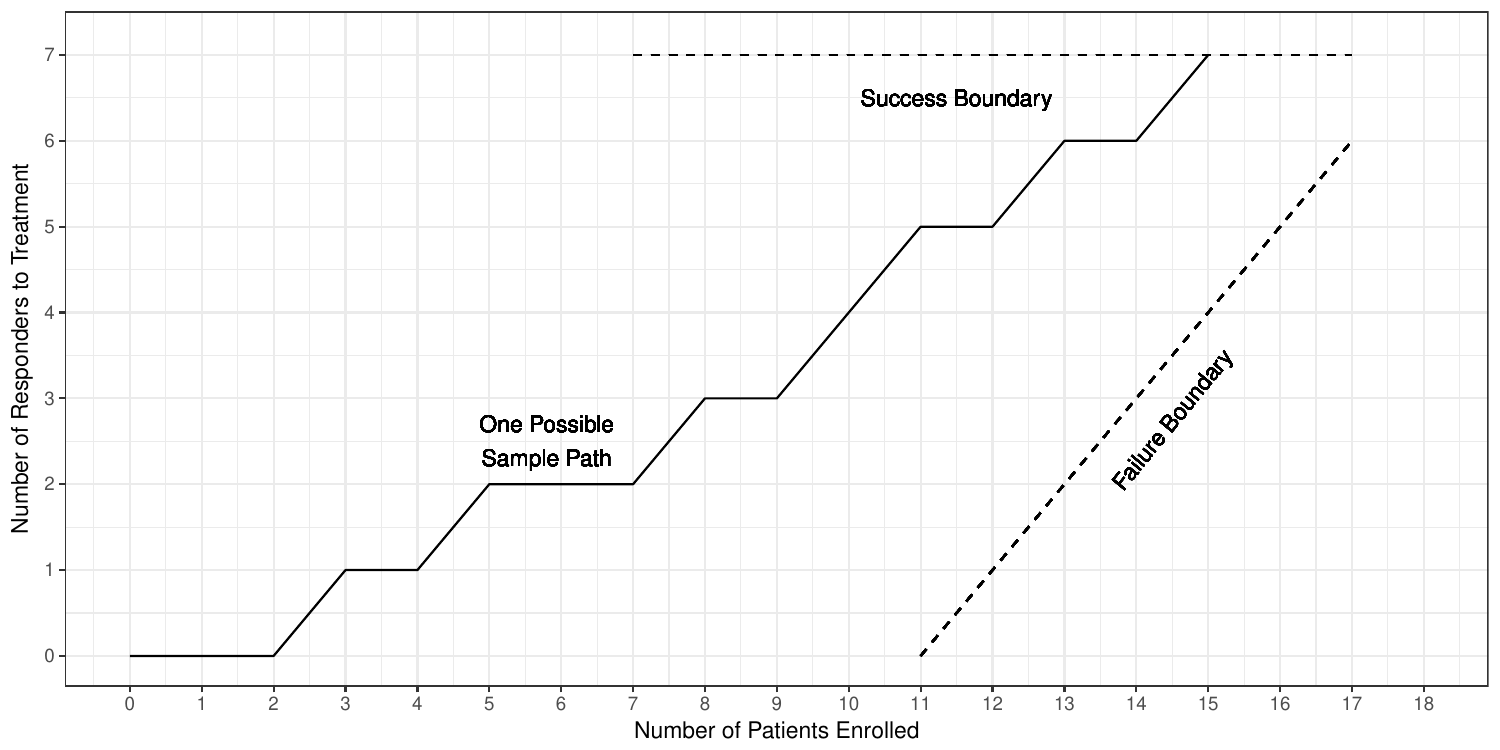}
\caption{
A hypothetical realization of the trial.
}
\label{fig:kane_viz}
\end{figure}

\pagebreak

\begin{figure}[htbp!]
\begin{center}
\includegraphics[width=\textwidth]{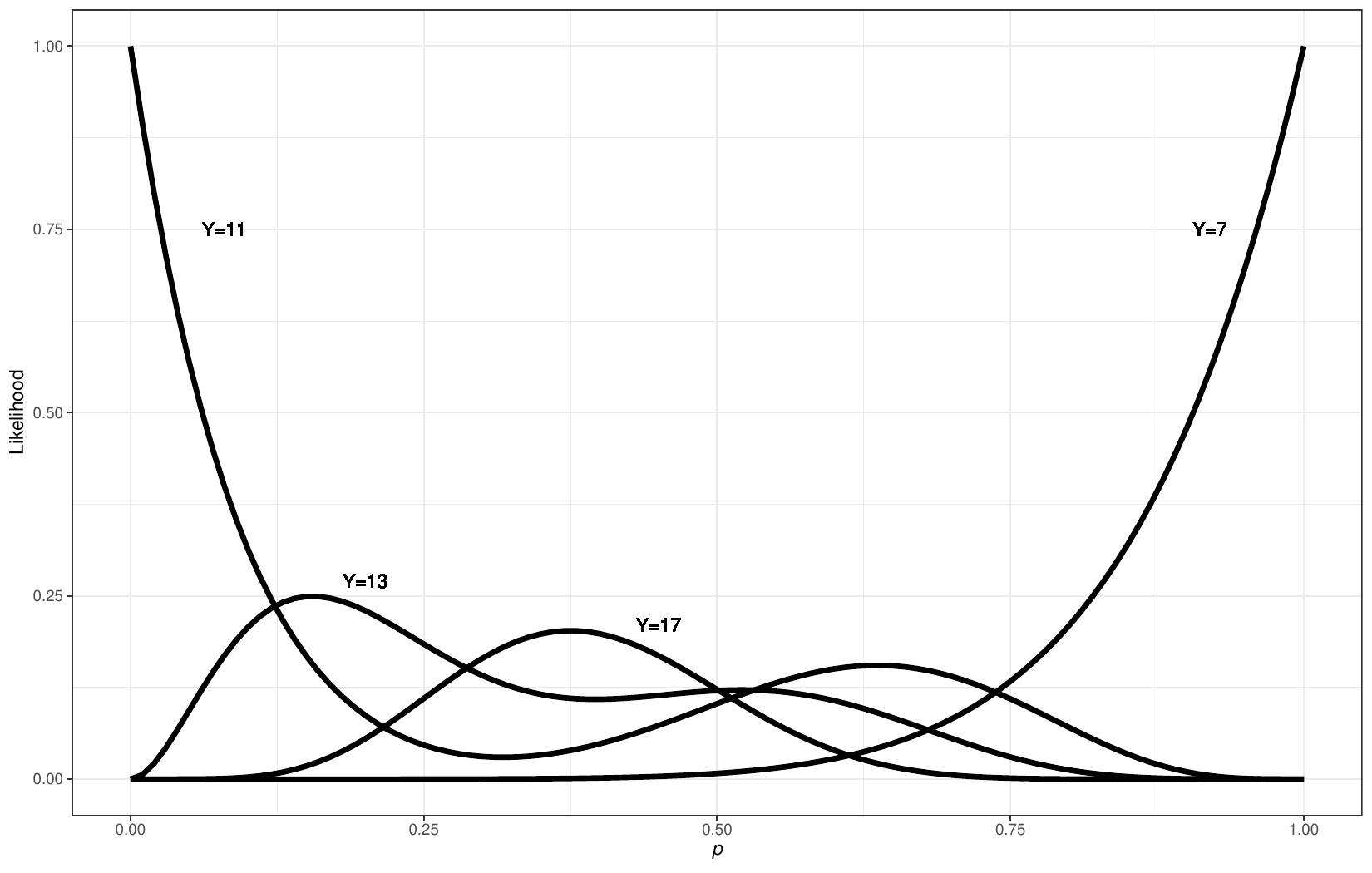}
\end{center}
\caption{Shapes of the likelihood function for given values of $Y$ with 
$s=7$ and $t=11$.
\label{likelihood.fig}}
\end{figure}

\pagebreak

\begin{figure}[htbp!]
\begin{center}
\includegraphics[width=0.8\textwidth]{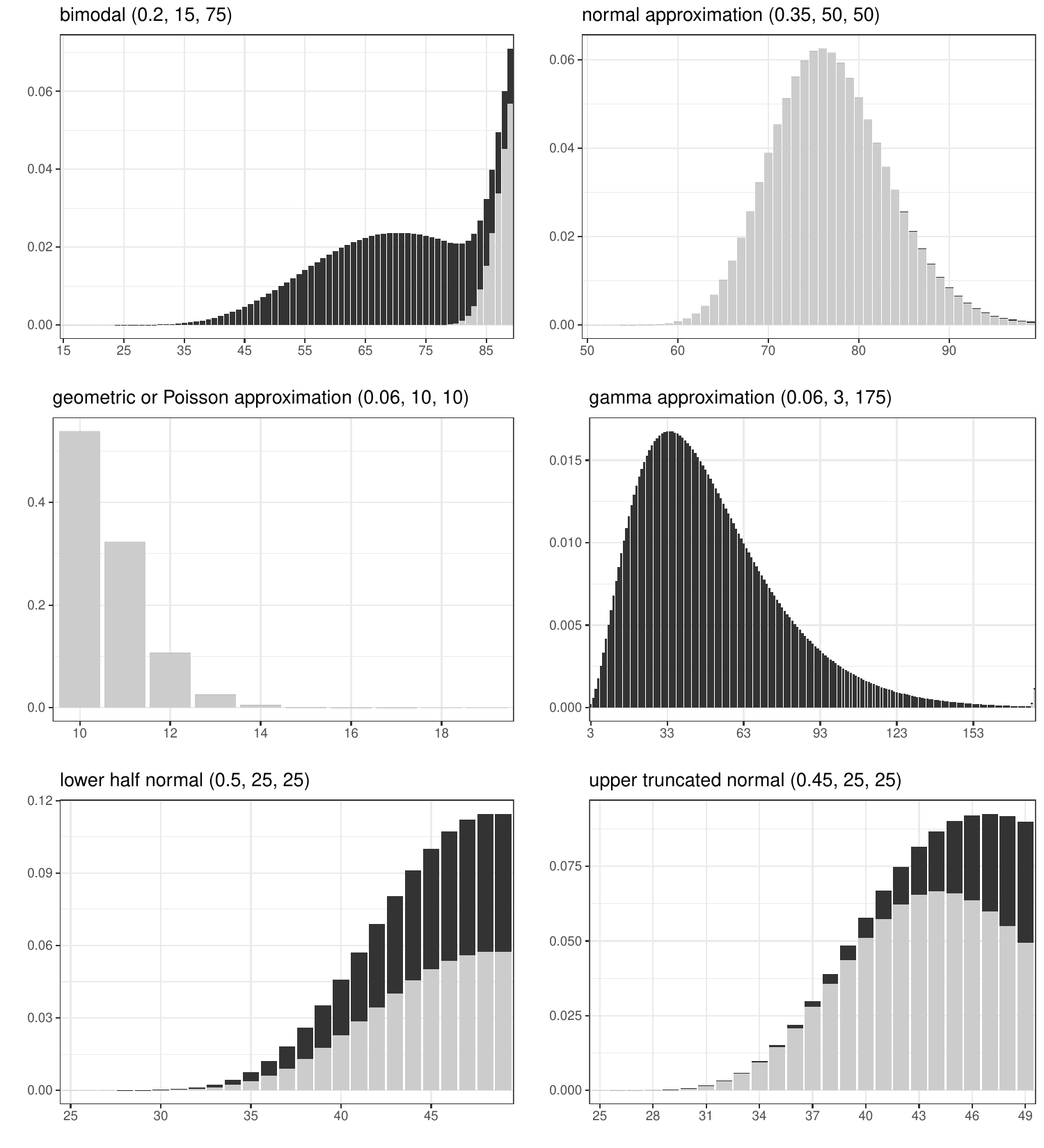}
\end{center}
\caption{Different shapes of the SNB mass function with parameters 
($p$, $s$, $t$) given. Black areas indicate the mass contributed by reaching
$s$ responders before $t$ non-responders. Grey indicates
mass contributed by reaching $t$ non-responders first. \label{shapes.fig}}
\end{figure}

\pagebreak

\begin{figure}[htbp!]
\centering
\subfloat[SNB distribution]{\includegraphics[width=0.5\textwidth]{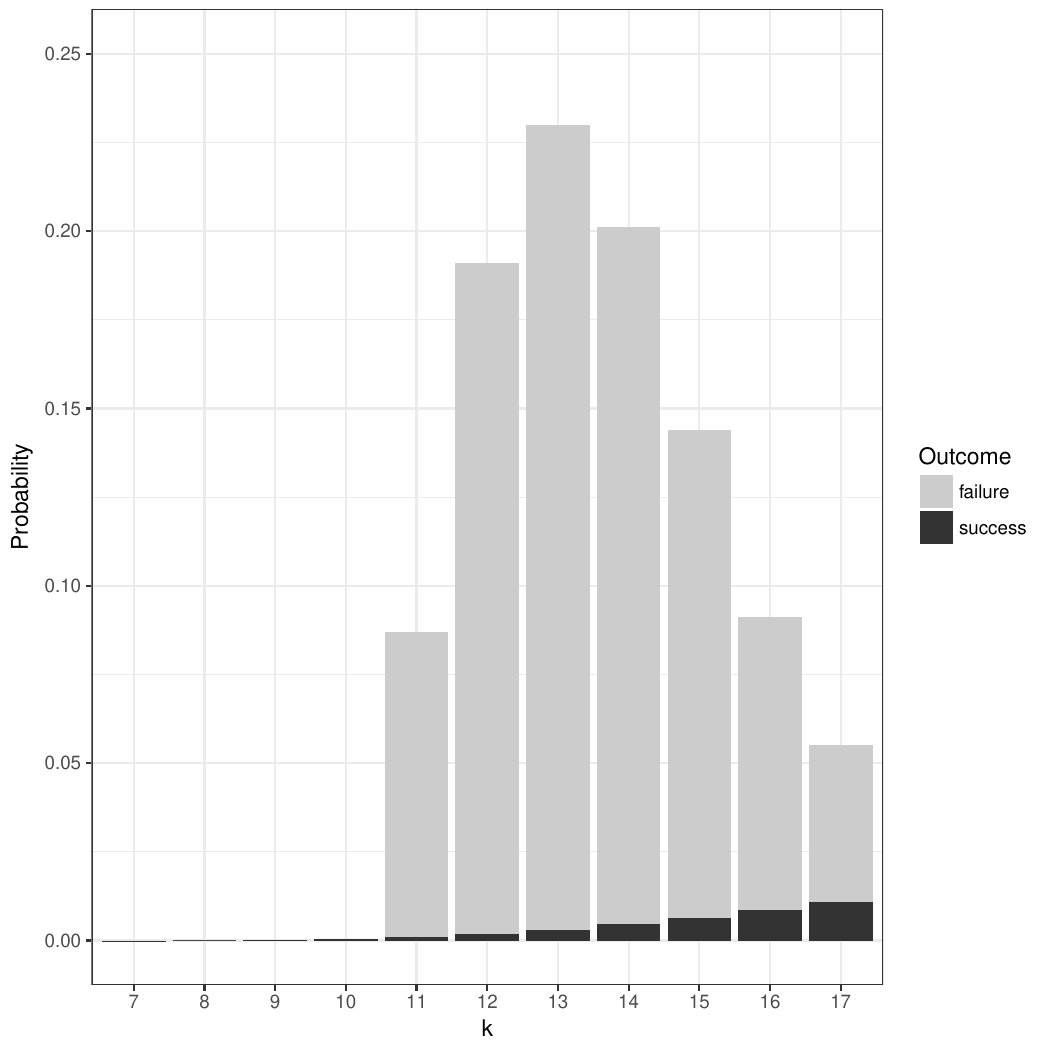}}
\hfill
\subfloat[binomial distribution]{\includegraphics[width=0.5\textwidth]{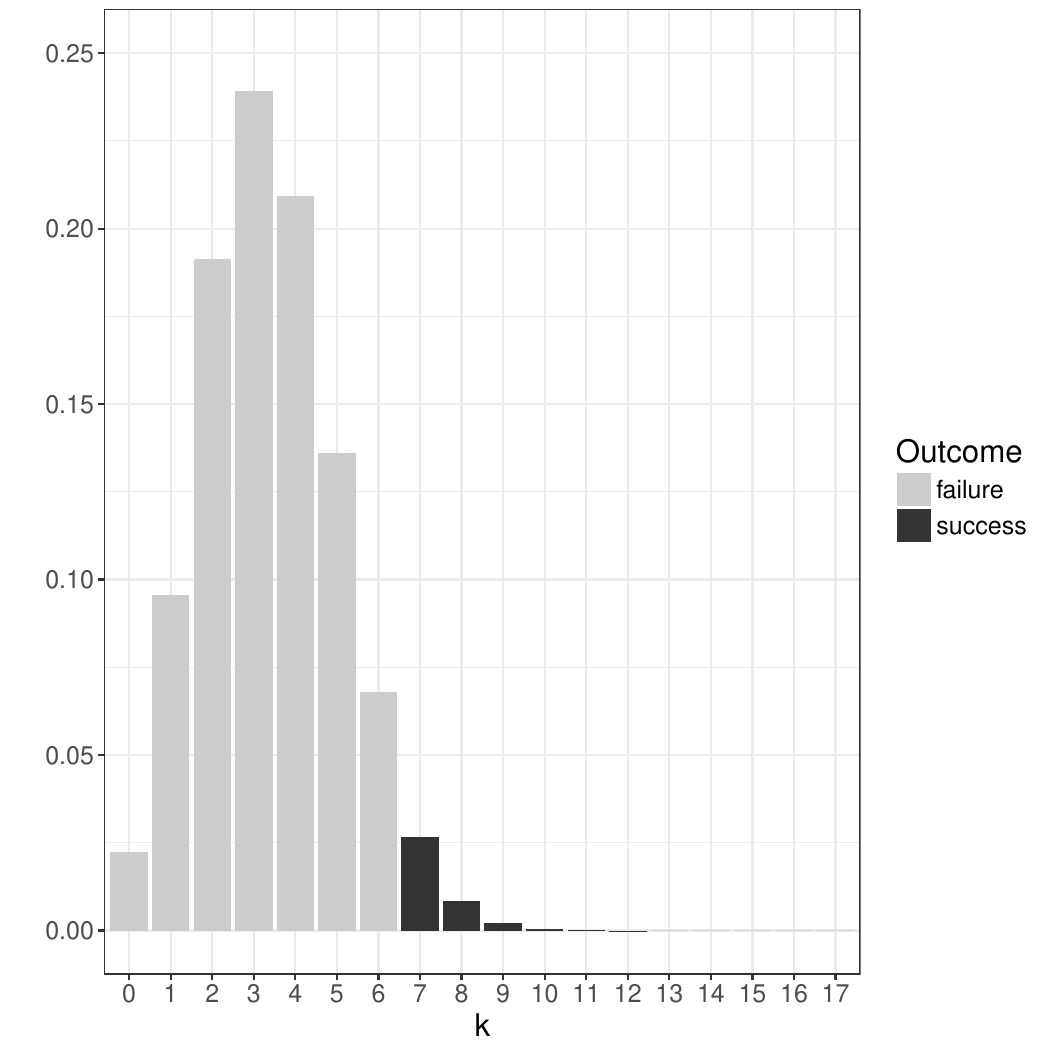}}
\caption{
SNB(0.2, 7, 11) with mass contributed from 
7 responders (black) or 11 non-responders (grey) along with 
Binomial(0.2, 17) with at least 2 responders (black) or fewer (grey).
}
\label{fig:snb_bin_compare}
\end{figure}

\pagebreak

\begin{figure}[htbp!]
\centering
\includegraphics[width=\textwidth]{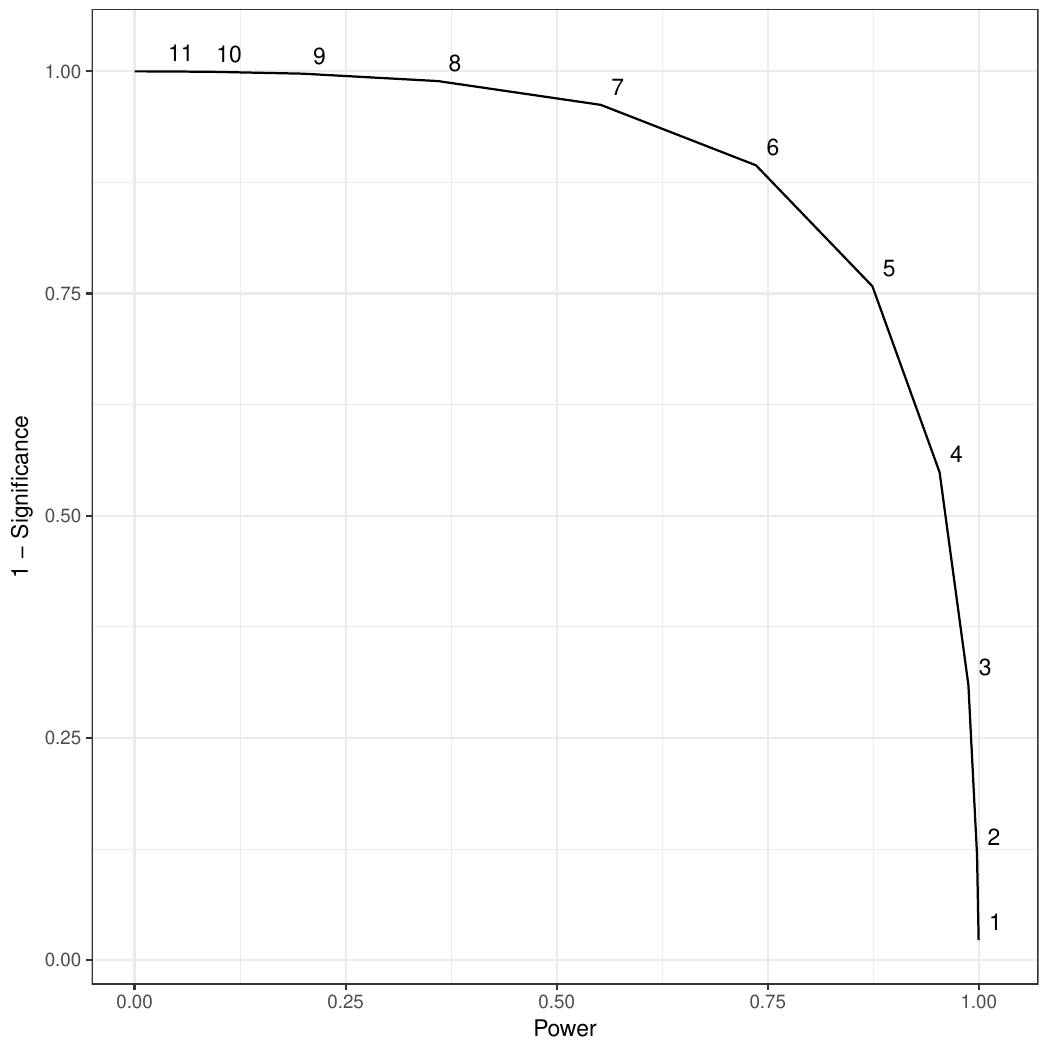}
\caption{
ROC curve of the all trial designs where the maximum number of patients is 17, $p=0.2$ under the null and $p=0.4$ under the alternative. The numerical values indicate the number of responses $s$ required reach the success endpoint.
}
\label{fig:all_sample_sizes_power_vs_signif}
\end{figure}

\pagebreak

\begin{figure}[htbp!]
\centering
\includegraphics[width=\textwidth]{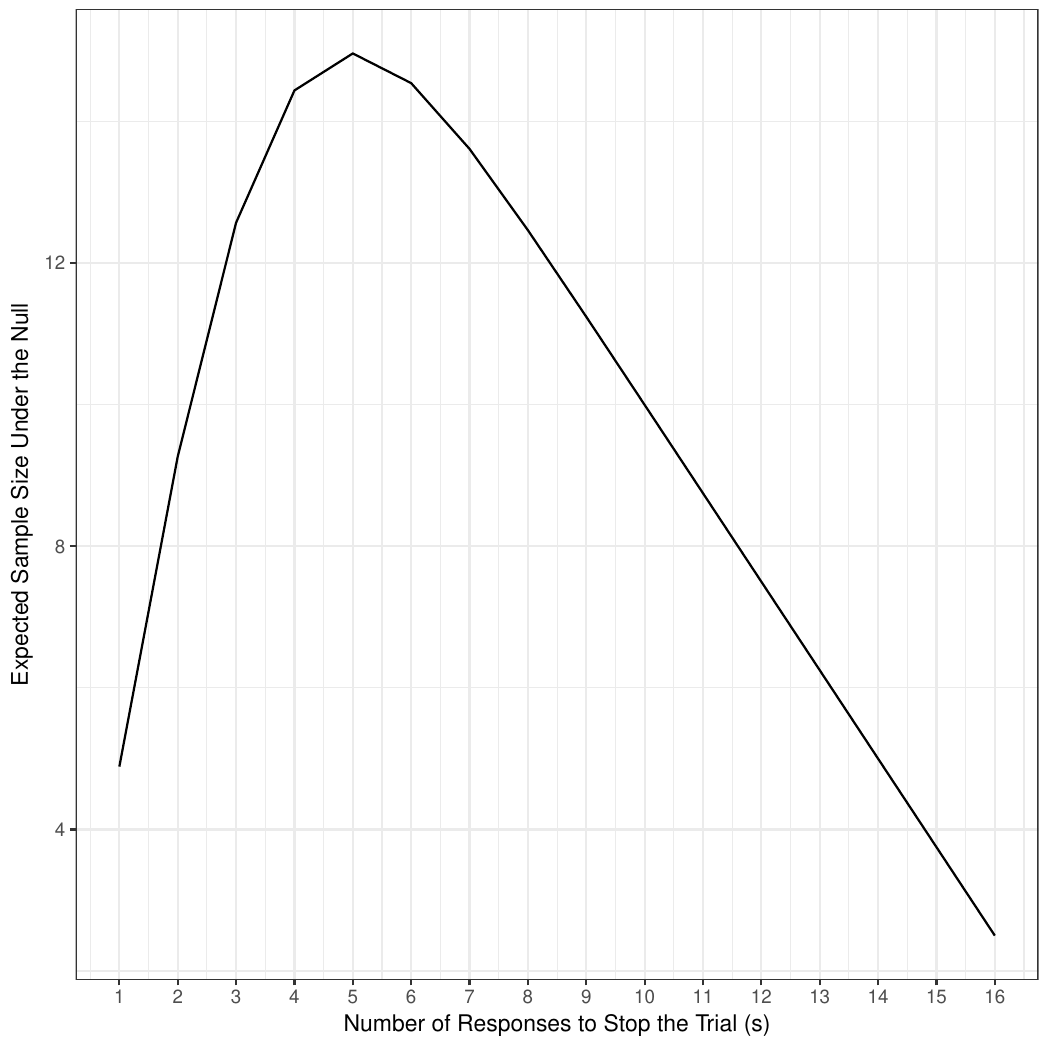}
\caption{
The expected sample size all trial designs where the maximum number
of patients is 17 and $p=0.2$.}
\label{fig:all_sample_sizes}
\end{figure}

\pagebreak

\end{document}